\documentclass{amsart} 
\usepackage{enumerate, amssymb, amsfonts, latexsym}
\usepackage{amscd, amsmath}
\usepackage{graphicx}
\newtheorem{theorem}{Theorem} 
\newtheorem{lemma}{Lemma}

\newtheorem{claim}{Claim}

\begin{document} 

\title{The Classical Stochastic Impulse Control Problem}
\author{Rohit Jain}

\begin{abstract}
In this paper we study regularity estimates for the solution to an obstacle problem arising in stochastic impulse control theory. We prove using elementary methods the known sharp $C_{loc}^{1,1}$ estimate for the solution. The new proof is also easily generalizable to stochastic impulse control problems with fully noninear operators. Moreover we obtain new regularity estimates for the free boundary in the classical case. 
\end{abstract}

\maketitle

\section{Introduction}
We consider an implicit constraint obstacle problem arising in impulse control theory. Stochastic impulse control problems (\cite{BL82}, \cite{L73}, \cite{M76}, \cite{F79}) are control problems that fall between classical diffusion control and optimal stopping problems. In these problems the controller is allowed to instantaneously move the state process by a certain amount every time the state exits the non-intervention region. This allows for the controlled process to have sample paths with jumps. There is an enormous literature studying stochastic impulse control models and many of these models have found a wide range of applications in electrical engineering, mechanical engineering, quantum engineering, robotics, image processing, and mathematical finance. Some classical references are \cite{L73}, \cite{F79}, \cite{BL82}. A key operator in stochastic impulse control problems is the intervention operator, 
\\
\begin{equation}
\textnormal{M}u(x) = \inf_{\xi \geq 0} (u(x+\xi) + 1).
\end{equation}
\\
The operator represents the value of the control policy that consists of taking the best immediate action in state $x$ and behaving optimally afterwards. Since it is not always the case that the optimal control requires intervention at $t = 0$, this leads to the quasi-variational inequality, 
\\
\begin{equation}
u(x) \leq \textnormal{M}u(x) \; \; \forall x \in \mathbb{R}^{n}.
\end{equation}
\\
From the analytic perspective one obtains an obstacle problem where the obstacle depends implicitly and nonlocally on the solution. More precisely we can consider the classical stochastic impulse control problem as a boundary value problem,
\\
\begin{equation}
   \begin{cases}
         Lu \leq f(x)& \forall x \in \Omega.\\
         u(x) \leq \textnormal{M}u(x) & \forall x \in \Omega.\\
         u = 0 & \forall x \in \partial \Omega.        
   \end{cases}
\end{equation}
\\
Here we let, $Lu \equiv -\sum_{i,j = 1}^{n} a_{ij}(x) \frac{\partial^{2}u}{\partial x_{i} \partial x_{j}} + \sum_{i = 1}^{n}b_{i}(x) \frac{\partial u}{\partial x_{i}} + c(x)u$ with suitable regularity assumptions on the data and,
\begin{equation}
\textnormal{M}u(x) = 1 +  \inf_{\stackrel{\xi \geq 0}{x + \xi \in \bar{\Omega}}}(u(x + \xi)).
\end{equation} 

In this paper, we present a new proof for the sharp $C_{loc}^{1,1}(\Omega)$ estimate for the solution to (1.3). We point out that the sharp $C^{1,1}_{loc}$ estimate has been previously obtained (see \cite{CF79a}, \cite{CF79b}).  As a corollary of our proof we also obtain a direct proof of the fact that the nonlocal obstacle, $\textnormal{M}u(x)$ is $C^{1,1}_{loc}$ on the contact set $\{u = \textnormal{M}u \}$. Since the obstacle depends on the solution, the strategy is to improve the regulariity of the solution and use it to improve the regularity of the obstacle. We start by first proving continuity of the solution and then proceeding to prove a semiconcavity estimate for the obstacle. In the following section we use the semiconcavity of the obstacle and the superhamonicity of the solution to produce the $C^{1,1}$ estimate. The new idea to prove the $C^{1,1}$ estimate in the classical case has been subsequently generalized to the fully nonlinear case with more general semiconcave obstacles \cite{J15B}.
\\
\\
In the last section we study regularity estimates for the free boundary $\partial \{u < \textnormal{M}u \}.$ We first observe that the set of free boundary points can be structured according to where $\inf u(x + \xi)$ is realized. If the infimum is realized in the interor of the positive cone then we conclude that the obstacle is locally constant in a neighborhood of a free boundary point. This gives us regularity estimates of the free boundary at regular points and singular points as defined in the classical obstacle problem. If the infimum is realized on the boundary of the cone then under the assumption that $f$ is analytic we conclude that the free boundary is contained in a finite collection of $C^{\infty}$ submanifolds. In particular we prove the following theorem,
\begin{theorem}
Consider the classical stochastic impulse control problem
\\
\begin{equation}
   \begin{cases}
         Lu \leq f(x)& \forall x \in \Omega.\\
         u(x) \leq \textnormal{M}u(x) & \forall x \in \Omega.\\
         u = 0 & \forall x \in \partial \Omega.        
   \end{cases}
\end{equation} 

Assume that all coefficients in $L$ are analytic, $f$ is analytic and $f(x) \leq f(x + \xi) \; \; \forall \xi \geq 0$. Then it follows that, $\partial \{u < \textnormal{M}u \} = \Gamma^{r}(u) \cup \Gamma^{s}(u) \cup \Gamma^{d}(u)$ where,
\\
\\
1. $\forall x_{0} \in \Gamma^{r}(u)$ there exists some appropriate system of coordinates in which the coincidence set $\{u = \textnormal{M}u \}$ is a subgraph $\{x_{n} \leq g(x_{1}, \dots, x_{n-1}) \}$ in a neighborhood of $x_{0}$ and the function $g$ is analytic.
\\
\\
2. $\forall x_{0} \in \Gamma^{s}(u)$, $x_{0}$ is either isolated or locally contained in a $C^{1}$ submanifold.
\\
\\
3. $\Gamma^{d}(u) \subset \Sigma(u)$ where $\Sigma(u)$ is a finite collection of $C^{\infty}$ submanifolds.
\end{theorem}

\emph{Acknowledgements} I would like to express my sincerest gratitude and deepest apprecation to my thesis advisors Professor Luis A. Caffarelli and Professor Alessio Figalli. It has been a truly rewarding experience learning from them and having their guidance. 

\section{Basic Definitions and Assumptions}

Let $\Omega \subseteq \mathbb{R}^{n}$ be a bounded domain with $C^{2,\alpha}$ boundary $\partial \Omega$. Assume $c(x) \geq c_{0} > 0$; $a_{ij}$, $b_{i}$, $c$, $\in C^{2+\alpha}(\bar{\Omega})$ for $0 < \alpha <1$, and the matrix $(a_{ij})$ is positive definite for all $x \in \bar{\Omega}$. Furthermore let $f \in C^{\alpha}(\bar{\Omega})$. For any $\xi = (\xi_{1}, \dots , \xi_{n})$ we let $\xi \geq 0$ denote $\xi_{i} \geq 0 \; \; \forall i$. Consider,
\begin{equation} 
Lu \equiv -\sum_{i,j = 1}^{n} a_{ij}(x) \frac{\partial^{2}u}{\partial x_{i} \partial x_{j}} + \sum_{i = 1}^{n}b_{i}(x) \frac{\partial u}{\partial x_{i}} + c(x)u.
\end{equation}

Define the operator:
\begin{equation}
\textnormal{M}u(x) = 1 + \inf_{\stackrel{\xi \geq 0}{x + \xi \in \Omega}}u(x + \xi).
\end{equation}

We introduce the bilinear form $a(u,v)$ associated to our operator L,
\begin{equation}
a(u,v) = (Lu,v) \; \; \forall u,v \in C_{0}^{\infty}(\Omega).
\end{equation}

Furthermore assume that our bilinear form is coercive,
\begin{equation}
a(u,u) \geq \gamma (\|u\|_{W^{1,2}(\Omega)})^{2} \; \; \forall u \in W_{0}^{1,2}(\Omega), \; \; \gamma >0.
\end{equation}

We consider the quasi-variational inequality:
\begin{center}
$u \in W_{0}^{1,2}(\Omega) \; \; u \leq \textnormal{M}u,$
\end{center}
\begin{equation}
a(u,v-u) \geq (f,v-u) \; \; \forall v \in W_{0}^{1,2}(\Omega) \; \; v \leq \textnormal{M}u.
\end{equation}

We list a few properties of our operator $\textnormal{M}u$ that will be useful for the remaining parts of this chapter,
$$ u_{1} (x) \leq u_{2} \; a.e. \Rightarrow \textnormal{M}u_{1} (x) \leq \textnormal{M}u_{2} (x) \; a.e.$$
$$\textnormal{M}: L^{\infty} \to L^{\infty}.$$
$$\textnormal{M}: C(\bar{\Omega}) \to C(\bar{\Omega}).$$ 

Furthermore we assume that $f \geq -\frac{1}{c_{0}}$. This implies that the solution $\bar{u}$ to the variational equation $L\bar{u} = f$ in $\Omega$, $\bar{u} \in H^{1}_{0}(\Omega)$ satisfies the property $\bar{u} \geq -1$. This in particular implies that the set of solutions to $v \in H^{1}_{0}(\Omega) \; \; v \leq \textnormal{M}\bar{u}$ is nonempty. Without loss of generality we assume that $\bar{u} < 1$. 

\section{Existence and Uniqueness Theory}

We now proceed to prove the existence of a unique continuous solution to (3.5). We follow closely the proof in \cite{I95}.  

\begin{lemma} There exists a unique solution $u \in C(\Omega)$ of (3.5).
\end{lemma}

\begin{proof}
From standard elliptic theory we know that there exists a unique solution $u_{0} \in C(\Omega)$ of

\begin{equation}
   \begin{cases}
         a(u,v) = (f,u-v) & \forall x \in \Omega,\\
         u = 0 & \forall x \in \partial \Omega.        
   \end{cases}
\end{equation}

Since $\textnormal{M}u_{0}$ is continuous we know from the theory of variational inequalities that there exists a unique solution $u_{1} \in C(\Omega)$ of
\begin{equation}
   \begin{cases}
         a(u,v) \geq (f,u-v) & \forall x \in \Omega,\\
         u\leq \textnormal{M}u_{0} & \forall x \in \Omega,\\
         u = 0 & \forall x \in \partial \Omega.        
   \end{cases}
\end{equation}

Moreover for $n = 2,3, \ldots$  we obtain $u_{n} \in C(\Omega)$ satisfying,
\begin{equation}
   \begin{cases}
         a(u,v) \geq (f,u-v) & \forall x \in \Omega,\\
         u\leq \textnormal{M}u_{n-1} & \forall x \in \Omega,\\
         u = 0 & \forall x \in \partial \Omega.        
   \end{cases}
\end{equation}

Since $u_{1}$ is a subsolution of (11), by the comparison principle, we know that $u_{1} \leq u_{0}$. We also know that $0$ is a subsolution of (12), hence the comparison implies that $0 \leq u_{1}$. Moreover it follows from the properties of $\textnormal{M}u$ that $0 \leq \textnormal{M}u_{1} \leq \textnormal{M}u_{0}$. This implies in particular that $u_{2}$ is an admissable subsolution to (12). Arguing as before we see that $0 \leq u_{2} \leq u_{1}$. We can continue this process and obtain a sequence of functions
\begin{equation}
0 \leq \ldots \leq u_{n} \leq \ldots \leq u_{1} \leq u_{0}.   
\end{equation}

Now we look to prove an upper bound on the sequence. Consider $\mu \in (0,1)$ such that $\mu \|u_{0}\|_{C(\Omega)} \leq 1$. Assume there exists $\theta_{n} \in (0,1]$ such that $\forall n \in \mathbb{N}$,
\begin{equation}
u_{n} - u_{n+1} \leq \theta_{n} u_{n}.  
\end{equation}

We claim that this implies
\begin{equation}
u_{n+1} - u_{n+2} \leq \theta_{n}(1-\mu) u_{n+1}.  
\end{equation}

With this claim we are able to almost conclude the proof of the theorem. In particular the positivity of $u_{n}$ implies that $u_{1} - u_{2} \leq u_{2}$. We can set $\theta_{1} = 1$. Moreover from (16) it follows that  $u_{2} - u_{3} \leq (1-\mu) u_{2}$. Hence $\theta_{2} = (1-\mu)$. Therefore setting $\theta_{n} = (1-\mu)^{n-1}$ we find
\begin{equation}
u_{n+1} - u_{n+2} \leq (1-\mu)^{n} u_{n+1} \leq (1- \mu)^{n} \|u_{0}\|_{C(\Omega)}.  
\end{equation}

Combining (17) with (14) we see that there exists a function $u \in C(\Omega)$ such that $\|u_{n} - u\|_{C(\Omega)} \to 0$ as $n \to \infty$. Moreover  from the estimate $\|\textnormal{M}u - \textnormal{M}v\|_{C(\Omega)} \leq \|u - v\|_{C(\Omega)}$ it follows that $u$ is a solution to the classical stochastic impulse control problem. Hence we are reduced to proving (16) and establishing uniqueness of the solution.
By the concavity of $\textnormal{M}u$ and (15) it follows,
\begin{equation}
\psi = (1 - \theta_{n})\textnormal{M}u_{n} + \theta_{n} \leq (1-\theta_{n})\textnormal{M}u_{n} + \theta_{n}\textnormal{M}0 \leq \textnormal{M}(1-\theta_{n}u_{n}) \leq \textnormal{M}u_{n+1}. \tag{*}
\end{equation}

We consider the continuous solutions to the following obstacle problems. Let $w \in C(\Omega)$ solve,
\begin{equation}
   \begin{cases}
         a(u,v) \geq (f,u-v) & \forall x \in \Omega.\\
         u\leq \psi & \forall x \in \Omega.\\
         u = 0 & \forall x \in \partial \Omega.        
   \end{cases}
\end{equation}

Let $z \in C(\Omega)$ solve,
\begin{equation}
   \begin{cases}
         a(u,v) \geq (f,u-v) & \forall x \in \Omega.\\
         u\leq 1 & \forall x \in \Omega.\\
         u = 0 & \forall x \in \partial \Omega.        
   \end{cases}
\end{equation}

From (*) and the comparision theorem for variational inequalities it follows that $w \leq u_{n+2}$. Moreover it follows that $\theta_{n}z$ solves,
\begin{equation}
   \begin{cases}
         a(u,v) \geq (f,u-v) & \forall x \in \Omega.\\
         u\leq \theta_{n} & \forall x \in \Omega.\\
         u = 0 & \forall x \in \partial \Omega.        
   \end{cases}
\end{equation}

Observing that $\psi \geq \theta_{n}$, it follows from comparision that $\theta_{n}w \geq \theta_{n}z$. Next we observe that $(1-\theta_{n})u_{n+1}$ is a subsolution and $(1-\theta_{n})w$ is a solution of the following obstacle problem,
\begin{equation}
   \begin{cases}
         a(u,v) \geq (f,u-v) & \forall x \in \Omega.\\
         u\leq (1-\theta_{n})\psi & \forall x \in \Omega.\\
         u = 0 & \forall x \in \partial \Omega.        
   \end{cases}
\end{equation}

Hence we find, $(1-\theta_{n})u_{n+1} \leq (1-\theta_{n})w$ . Putting this together we obtain,
\[
(1-\theta_{n})u_{n+1} + \theta_{n}z \leq (1-\theta_{n})w + \theta_{n}w = w \leq u_{n+2}. \tag{**}
\]
Recall that $\forall n$, $\mu u_{n+1} \leq 1$. This implies that $\mu u_{n+1}$ is a subsolution of (19). So in particular, $\mu u_{n+1} \leq z$. Putting this into (**) we obtain our desired estimate (16),
$$u_{n+1} - u_{n+2} \leq \theta_{n}(1-\mu) u_{n+1}.$$
Finally to prove uniqueness, suppose $u$ and $\bar{u}$ are distinct solutions. The positivity of the solution implies $u - \bar{u} \leq u$. Hence arguing as above we find $u - \bar{u} \leq (1- \mu)^{n} u$, for all $n \geq 0$. Letting $n \to \infty$ we find that $u - \bar{u} \leq 0$. Interchanging $u$ and $\bar{u}$ we conclude $u = \bar{u}$. 
\end{proof}

Using the improved regularity on the solution $u$, we now proceed to prove that the obstacle $Mu(x)$ is semi-concave with semi-concavity modulus $ \omega(r) = Cr^{2}$. We state the following theorem which is proven in more generality in \cite{CF79a}, \cite{J15B}. 

\begin{theorem} Let $\varphi(x) \in C^{1,1}(\Omega)$, strictly positive, bounded, and decreasing in the positive cone $\xi \geq 0$. Then the obstacle $$Mu(x) = \varphi(x) + \inf_{\stackrel{\xi \geq 0}{x + \xi \in \Omega}}u(x + \xi)$$ 
is locally semi-concave with a semi-concavity modulus $\omega(r) = Cr^{2}$.
\end{theorem}

\section{Optimal Regularity for the Stochastic Impulse Control Problem}

In the previous section we proved that the unique bounded solution to the classical stochastic impulse control problem is continuous and that our implicit constraint obstacle is locally semi-concave. We now consider the sharp $C^{1,1}$ estimate for the solution.

\begin{theorem}  Let $u$ be the unique continuous solution of the classical stochastic impulse control problem. Then
$u \in C^{1,1}_{loc}(\Omega)$. 
\end{theorem}

We set $Mu(x) = \varphi_{u}(x)$. Recall that a function $v$ is semi-concave with semi-concavity modulus $\omega(r)$ if a vector $p \in \mathbb{R}^{n}$ belongs to $D^{+}v(x)$ if and only if $v(y) - v(x) - \langle p,y-x \rangle \leq \omega(|x - y|)$. Fix $x_{0} \in \{u = \varphi_{u}\}$. Define the linear part of the obstacle, $L_{x_{0}}(x) =  \varphi_{u}(x_{0}) + \langle p,x-x_{0} \rangle$. We consider
\\
\[
w(x) = u(x) - L_{x_{0}}(x).
\]
\\
We observe that in $B_{r}(x)$, $w(x)$ has a modulus of semi-concavity $\omega(r) = Cr^{2}$, i.e. $w(x) \leq Cr^{2}$. We now state our main lemma.

\begin{lemma} There exists universal constants $K, C > 0$, such that $\forall x \in B_{r/4}(x_{0})$, 
\begin{equation}
-K \leq \Delta w \leq C.
\end{equation}
\end{lemma}

Before proving this lemma we make a few observations. Fix $\varPhi \in C^{\infty}_{0}(B_{\frac{r}{2}}(x_{0}))$.  We recall the following fact from the theory of distributions: If u is a negative distribution in X with $u(\varPhi) \leq 0$ for all non-negative $\varPhi \in C^{\infty}_{0}(X)$, then u is a negative measure. In particular we have,
\begin{equation}
0 \geq \int_{B_{\frac{r}{2}}}\varPhi \; d\mu = \int_{B_{\frac{r}{2}}} \Delta u \; \varPhi.
\end{equation}
We consider $\forall \rho < \frac{r}{2}$,
\begin{equation}
\frac{\mu(B_{\rho}(x_{0}))}{|B_{\rho}(x_{0})|}  = \frac{1}{\alpha(n)\rho^{n}} \int_{B_{\rho}}d\mu =  \frac{1}{\alpha(n)\rho^{n}} \int_{B_{\rho}} \Delta u.
\end{equation}
A straightforward application of the Gauss-Green Formula gives to us the following identity,
\begin{equation}
\frac{1}{\alpha(n)\rho^{n}} \int_{B_{\rho}} \Delta w = \frac{n}{\rho}\frac{d}{d\rho}\varPsi(\rho).
\end{equation}
Where $\varPsi(\rho) = \frac{1}{n\alpha(n)\rho^{n-1}} \int_{\partial B_{\rho}} w$. Before proving the main lemma we will first prove the following claim.

\begin{claim} Let $w = u - L_{x_{0}}$ be defined as before. Then for some universal constant $K(n) > 0$,
\[
\frac{n}{\rho}\frac{d}{d\rho}\varPsi(\rho) \geq -K.
\]
\end{claim}

\begin{proof}
We expand the derivative and compute.
\[
\begin{split}
\frac{n}{\rho}\frac{d}{d\rho}\varPsi(\rho)  & = \frac{n}{\rho} \frac{1-n}{n\alpha(n)\rho^{n}} \int_{\partial B_{\rho}(x_{0})} w(y) dS(y) + \frac{n}{n\alpha(n)\rho^{n}} \frac{d}{d\rho} \int_{\partial B_{\rho}(x_{0})} w(y) dS(y) \\
 & = \frac{n}{\rho} \frac{n-1}{n\alpha(n)\rho^{n}} \int_{\partial B_{\rho}(x_{0})} -w(y) dS(y) + \frac{1}{\alpha(n)\rho^{n}} \frac{d}{d\rho} \rho^{n-1} \int_{\partial B_{1}(0)} w(x_{0} + \rho z) dS(z) \\
& =  \frac{n}{\rho} \frac{n-1}{n\alpha(n)\rho^{n}} \int_{\partial B_{\rho}(x_{0})} -w(y) dS(y) + \frac{\rho^{n-2} (n-1)}{\alpha(n)\rho^{n}} \frac{\rho^{n-1}}{\rho^{n-1}} \int_{\partial B_{1}(0)} w(x_{0} + \rho z) dS(z) \\ 
& + \frac{\rho^{n-1}}{\alpha(n)\rho^{n}}\frac{d}{d\rho} \int_{\partial B_{1}(0)} w(x_{0} + \rho z) dS(z) \\
& = \frac{n}{\rho} \frac{n-1}{n\alpha(n)\rho^{n}} \int_{\partial B_{\rho}(x_{0})} -w(y) dS(y) +  \frac{(n-1)}{\alpha(n)\rho^{n+1}} \int_{\partial B_{\rho}(x_{0})} w(y) dS(y) \\
& + \frac{\rho^{n-1}}{\alpha(n)\rho^{n}}\frac{d}{d\rho} \int_{\partial B_{1}(0)} w(x_{0} + \rho z) dS(z)
\end{split}
\]
\\
Now we proceed to estimate each integral. By the modulus of semi-concavity on the ball we have,
$$  \frac{n}{\rho} \frac{n-1}{n\alpha(n)\rho^{n}} \int_{\partial B_{\rho}(x_{0})} -w(y) dS(y) \geq \frac{n (n-1)}{\alpha(n) n \rho^{n+1}} \; |\partial B_{\rho}(x_{0})| \; (-C\rho^{2}) = -C(n^{2} - n).$$
By the mean value theorem for subharmonic functions we have,
$$ \frac{(n-1)}{\alpha(n)\rho^{n+1}} \int_{\partial B_{\rho}(x_{0})} w(y) dS(y) \geq \frac{(n-1)}{\alpha(n)\rho^{n+1}} \; w(x_{0}) = 0.$$
By the nondecreasing property for the average integral we have:
$$ \frac{\rho^{n-1}}{\alpha(n)\rho^{n}}\frac{d}{d\rho} \int_{\partial B_{1}(0)} w(x_{0} + \rho z) dS(z) \geq 0.$$
Hence for $K = C(n^{2} - n)$ we obtain the desired estimate.
\end{proof}

\begin{proof} (Lemma 3)
From the claim we obtain the estimate,
$$\frac{1}{\alpha(n)\rho^{n}} \int_{B_{\rho}} \Delta w \geq -K.$$
Moreover from (23) and the semi-concavity estimate from above we know,
$$ C \geq \frac{\mu(B_{\rho}(x_{0}))}{|B_{\rho}(x_{0})|} \geq -K.$$

Letting $\rho \to 0$ we find $\forall x \in B_{\frac{r}{4}}(x_{0})$,
$$ C \geq \Delta u(x) \geq -K.$$

\end{proof}

We now state and prove the sharp estimate for the solution.

\begin{theorem} Let $u$ be a solution to the classical stochastic impulse control problem. Then,
\\
\begin{equation}
\|u\|_{C^{1,1}(B_{r/4})} \leq C
\end{equation}
\\
\end{theorem}

\begin{proof}
We recall some basic notions and definitions for convenience. For further details refer to (\cite{CC95}). We say that $P$ is a parabaloid of opening $M$ whenever,
$$P(x) = l_{0} + l(x) \pm \frac{M}{2} |x|^{2}.$$
We define,
$$\overline{\Theta}(u,A)(x_{0}),$$
to be the infimum of all positive constants $M$ for which there is a conex parabaloid of opening $M$ that touches $u$ from above at $x_{0}$ in $A$. Similarly one can define the infimum of all positive constants $M$ for which there is a convex parabaloid of opening $-M$ that touches $u$ from below at $x_{0}$ in $A$,
$$\underline{\Theta}(u,A)(x_{0}).$$
We further define,
$$\Theta(u,A)(x_{0}) = \sup \{\overline{\Theta}(u,A)(x_{0}), \underline{\Theta}(u,A)(x_{0}) \} \leq \infty.$$
As before we fix $x_{0} \in \{u = Mu \}$. We consider the second incremental quotients of $u$ and $Mu$,
\\
\[
\begin{split}
 \Delta_{h}^{2}u(x_{0}) &  = \frac{u(x_{0} + h) + u(x_{0} - h) - 2 u(x_{0})}{|h|^{2}}.\\
 \Delta_{h}^{2}Mu(x_{0}) &  = \frac{Mu(x_{0} + h) + Mu(x_{0} - h) - 2 Mu(x_{0})}{|h|^{2}}.
\end{split}
\]
\\
We make the following observations,
\\
1. $\Delta_{h}^{2}u(x_{0}) \leq \Delta_{h}^{2}Mu(x_{0}).$
\\
2. $0 \leq \overline{\Theta}(u,B_{\rho})(x_{0}) = \overline{\Theta}(Mu,B_{\rho})(x_{0}) \leq C$.
\\
3. $0 \leq \underline{\Theta}(u,B_{\rho})(x_{0}) = \underline{\Theta}(Mu,B_{\rho})(x_{0}) \leq K$.
\\
\\
Putting the estimates together we obtain,
$$-K \leq -\underline{\Theta}(u,B_{\rho})(x_{0}) \leq \Delta_{h}^{2}u(x_{0}) \leq \Delta_{h}^{2}Mu(x_{0}) \leq \overline{\Theta}(Mu,B_{\rho})(x_{0}) \leq C.$$
In particular $\forall x \in B_{\rho}$,
$$-K  \leq -\underline{\Theta}(u,B_{\rho})(x) \leq \Delta_{h}^{2}u(x) \leq \overline{\Theta}(u,B_{\rho})(x) \leq C.$$
This follows from choosing $\forall x \in B_{\rho}$, the lower parabaloid and upper parabaloid to be respectively,
$$P_{1}(y) = u(x) + \langle p_{1}, y-x \rangle - \frac{K}{2} |y|^{2}.$$
$$P_{2}(y) = u(x) + \langle p_{2}, y-x \rangle + \frac{C}{2} |y|^{2}.$$
Hence we obtain,
$$\Theta(u, \epsilon) = \Theta(u, B_{\rho} \cap B_{\epsilon}(x))(x) \in L^{\infty}(B_{\rho}).$$
This implies,
$$\|D^{2}u\|_{L^{\infty}(B_{\rho})} \leq C.$$
In particular we obtain our desired estimate,
$$\|u\|_{C^{1,1}(B_{\rho})} \leq C.$$
\end{proof}

\section{Regularity Estimates for the Free Boundary}

In this section we prove a structural theorem for the free boundary $\Gamma = \partial \{u < \textnormal{M}u \}.$
\begin{theorem}
Consider the classical stochastic impulse control problem
\begin{equation}
   \begin{cases}
         \Delta u(x) \geq f(x)& \forall x \in \Omega,\\
         u(x) \leq \textnormal{M}u(x) = 1 + \inf_{\stackrel{\xi \geq 0}{x + \xi \in \Omega}}u(x + \xi)  & \forall x \in \Omega,\\
         u = 0 & \forall x \in \partial \Omega.        
   \end{cases}
\end{equation} 

Moreover assume that $f$ is analytic and $f(x) \leq f(x + \xi) \; \; \forall \xi \geq 0$. Then it follows that, $\partial \{u < \textnormal{M}u \} = \Gamma^{r}(u) \cup \Gamma^{s}(u) \cup \Gamma^{d}(u)$ where,
\\
1. $\forall x_{0} \in \Gamma^{r}(u)$ there exists some appropriate system of coordinates in which the coincidence set $\{u = \textnormal{M}u \}$ is a subgraph $\{x_{n} \leq g(x_{1}, \dots, x_{n-1}) \}$ in a neighborhood of $x_{0}$ and the function $g$ is analytic.
\\
2. $\forall x_{0} \in \Gamma^{s}(u)$, $x_{0}$ is either isolated or locally contained in a $C^{1}$ submanifold.
\\
3. $\Gamma^{d}(u) \subset \Sigma(u)$ where $\Sigma(u)$ is a finite collection of $C^{\infty}$ submanifolds.
\end{theorem}

\begin{proof}

Recall $\Sigma_{x} =  \{1 + u(x + \xi) = \textnormal{M}u(x)\}$ and  $\Sigma_{\geq x} = \{x + \xi \; : \; \xi \geq 0\}.$ We define the following sets
\\ 
\\
1. $\Sigma_{\geq x}^{0} = \{\xi \in \Sigma_{\geq x} \; | \; \xi_{i} > 0 \; \forall i = 1, \dots , n\}.$
\\
2. $\partial_{i} \Sigma_{\geq x} = \{\xi \in \Sigma_{\geq x} \; | \; \xi_{i} > 0 \; \textnormal{and} \; \xi_{k} = 0 \; \; \forall k = 1, \dots , i-1, i +1, \dots, n \}.$
\\
3. $\Sigma_{x}^{0} = \{\xi \in \Sigma_{x} \; | \; \xi \in \Sigma_{\geq x}^{0} \}.$
\\
4. $\partial_{i} \Sigma_{x} = \{\xi \in \Sigma_{x} \; | \; \xi \in \partial_{i} \Sigma_{\geq x} \}.$
\\
\\
We note that 
$$\Sigma_{\geq x} = \Sigma_{\geq x}^{0} \cup (\bigcup_{i}^{n} \partial_{i} \Sigma_{\geq x}),$$
$$\Sigma_{x} = \Sigma_{x}^{0} \cup (\bigcup_{i}^{n} \partial_{i} \Sigma_{x}).$$

Fix $x_{0} \in \partial \{u < \textnormal{M}u \}$ and let $\xi_{0}$ be the positive vector such that, 
$$\inf_{\stackrel{\xi \geq 0}{x_{0} + \xi \in \Omega}}u(x_{0} + \xi) = 1 + u(x_{0} + \xi_{0}).$$ 
\\
\textbf{Case 1}: $\xi_{0} \in \Sigma_{x_{0}}^{0}.$ Then it follows from Claim 4 in \cite{J15B}, that $\forall x \in B_{\frac{\delta}{2}}(x_{0}),$ $\xi_{0} \in \Sigma_{x}^{0}.$ In particular for a fixed constant $C$, $\textnormal{M}u = C$ in $B_{\frac{\delta}{2}}(x_{0}).$ Without loss of generality we take $C = 0.$ Furthermore it follows that at a contact point $x_{0}$ we have the following chain of inequalities,
$$f(x_{0}) \leq \Delta u(x_{0}) \leq \Delta \textnormal{M}u(x_{0}) \leq f(x_{0} + \xi_{0}).$$

In particular,
$$f(x_{0}) \leq \Delta u(x_{0}) \leq 0.$$

We make the following claim,
\begin{claim} $f(x_{0}) < 0.$
\end{claim}

\begin{proof}
Suppose by contradiction that $f(x_{0}) = 0.$ By analyticity of $f$, it follows that $\Omega = \{f > 0 \}$ satisfies an interior sphere condition. Hence, $\forall z \in \partial \Omega$ there exists, $y \in \Omega$ and open ball  $B_{r}(y)$ such that $\overline{B_{r}(y)} \cap \overline{\Omega} = \{z\}.$ In particular consider $z =x_{0}$ and $y = y_{0}$. Observe that $\forall x \in \overline{B_{r}(y_{0})} \setminus \{x_{0}\}$, it follows that $w = u - \textnormal{M}u < 0$ and $\Delta w = \Delta u = f > 0$. Hence by the Hopf Boundary point lemma, 
$$\frac{\partial w}{\partial \nu}(x_{0}) > 0.$$
But $w \in C^{1,1}(x_{0}).$ A contradiction.
\end{proof}

From the claim it follows that in a small neighborhood $B_{\eta}(x_{0})$, we can study the following problem, 
\begin{equation}
   \begin{cases}
         \Delta w(x) = f(x) < 0& \forall x \in \{w < 0 \} \cap B_{\eta}(x_{0}),\\
         w(x) \leq 0  & \forall x \in B_{\eta}(x_{0}),\\
         w \in C^{1,1} & \forall x \in \overline{B_{\eta}(x_{0}) }       
   \end{cases}
\end{equation} 

Hence $w$ is a normalized solution and the conclusion follows for,

Finally to conclude we define,
$$\Gamma^{r}(u) = \{x \in \Gamma \; | \; \Sigma_{x}^{0} = \Sigma_{x} \; \textnormal{and} \; x \; \textnormal{is a \textbf{Regular Point}} \}.$$
$$\Gamma^{s}(u) = \{x \in \Gamma \; | \; \Sigma_{x}^{0} = \Sigma_{x} \; \textnormal{and} \; x \; \textnormal{is a \textbf{Singular Point}} \}.$$
\textbf{Case 2}: $\xi_{0} \in \partial_{i} \Sigma_{x_{0}}.$ We consider the set 
$$\Sigma(u) = \bigcup_{i}^{n} \{u_{x_{i}} = 0 \} \times \mathbb{R}^{n-1}.$$ 
By analyticity of $f$ it follows that $\{u_{x_{i}} = 0\}$ is a finite set $\forall i = 1,\dots,n.$ Hence $\Sigma(u)$ is a finite collection of hyperplanes $\{l_{j} \}_{j = 1}^{k} \subset \mathbb{R}^{n-1}.$ We define 
$$\Gamma^{d}(u) = \{x \in \Gamma(u) \; | \; \exists \bar{\xi} \in \partial_{i} \Sigma_{x} \}.$$
Finally to conclude we observe,
$$\Gamma^{d}(u) \subset \Sigma(u).$$  
\end{proof}

\end{document}